\documentclass[10pt]{amsart}
\usepackage[centertags]{amsmath}
\usepackage{amsxtra}
\usepackage[top=1in,bottom=1in,left=1in,right=1in]{geometry}
\usepackage{amssymb}
\usepackage{amsthm}
\usepackage{enumerate}
\usepackage{newlfont}
\newtheorem{theorem}{Theorem}
\newtheorem{definition}{Definition}[section]
\newtheorem{lemma}{Lemma}
\newtheorem{pro}[definition]{Proposition}
\newtheorem{corollary}[definition]{Corollary}

\numberwithin{equation}{section}

\usepackage{amsfonts}
\usepackage{mathrsfs}
\usepackage{amsmath}
\usepackage{amscd}
\usepackage[all,cmtip]{xy}

\usepackage{dsfont}
\usepackage{tikz}
\usepackage{tikz}
\usepackage{pgflibraryarrows}
\usepackage{pgflibrarysnakes}

\begin{document}

\title{\textbf{On $W_2$-Lifting of Frobenius of Algebraic Surfaces}}
\date{September 2, 2013}
\author{He Xin}
\address{Institut f\"{u}r Mathematik, Universit\"{a}t Mainz, Mainz, 55099, Germany}
\email{hexin1733@gmail.com}
\maketitle

\begin{abstract} We completely decide which minimal algebraic surfaces in positive characteristics allow a lifting of their Frobenius over the truncated Witt rings of length 2.
\end{abstract}
\fontfamily{cmr}
\thispagestyle{empty}

\section{Introduction}
\noindent To have Frobenius morphism makes schemes in characteristic $p$ distinct from those in characteristic 0. For a scheme in characteristic $p$ of dimension $r$, the Frobenius morphism defined on it is a finite morphism of degree $p^r$. However, if the base field is of characteristic 0, Beauville proved in \cite{beauville2001endomorphisms} that a smooth hypersurface in $\mathbb{P}^m$ of degree $d$
has no endomorphism of degree larger than 1 provided $m,d\geq3$.

\medskip One can show easily Beauville's result combined with Zariski's main theorem implies the following corollary. Suppose we are given a smooth hypersurface in mixed characteristic satisfying the numerical conditions above, then the Frobenius morphism on the special fiber cannot be lifted to a finite endomorphism of the whole hypersurface. In spite of this implication, one can ask further whether the Frobenius morphism can be lifted to the truncated Witt rings. To be more precise, let $X$ be a projective smooth hypersurface over an algebraically closed field $k$ with positive characteristic and $n\geq2$, is there a flat lifting $X_n$ of $X$ over $W_n(k)$ and a morphism $F_{X_n}:X_n\rightarrow X_n$ such that the restriction of $F_{X_n}$ to $X$ coincides with $F_X$?

\medskip It is a general belief that such liftings rarely exist, however it is difficult to verify such judgement for a given variety even when $n=2$. If $X$ a projective smooth curve, it is easy to prove the liftability of $F_X$ over $W_2(k)$ implies the genus of $X$ cannot be $\geq2$. A similar discussion given in \cite{dupuy2012positivity} shows that there exist no liftings of Frobenius to mixed characteristic for a variety with Kodaira dimension $\geq1$. For varieties with negative Kodaira dimension, the only known proved cases seem to be certain classes of flag varieties \cite{buch1997frobenius} based on Bott non-vanishing theorems for these varieties. It is noticeable that the result in \cite{buch1997frobenius} strengthens earlier results in \cite{paranjape1989self}, where such flag varieties over $p$-adic numbers are proved to have no endomorphisms lifting the Frobenius.  On the positive side, the only known varieties with liftable Frobenius seem to be ordinary varieties with trivial tangent bundles \cite{mehta1987varieties} and toric varieties \cite{buch1997frobenius}.

\medskip It is the aim of this paper to investigate this problem for minimal algebraic surfaces. We obtain a complete answer to the above question for $n=2$ in this situation.

\begin{theorem}\label{mt} Let $X$ be a smooth projective minimal algebraic surface over a field $k$ of characteristic $p$, then the Frobenius of $X$ is liftable over $W_2(k)$ iff $X$ belongs to one of the following classes.
\begin{enumerate}
\item $\kappa(X)=0$
\begin{enumerate}
\item Ordinary abelian surfaces,\\
\item\vspace{-10pt} Ordinary hyperelliptic surfaces of type a), b), c), or d) if $p\neq2, 3$ and type a) if $p=2$ or $3$ such that $\omega^{\otimes(p-1)}_{X/k}\cong\mathcal{O}_X$.\\
\end{enumerate}
\item\vspace{-5pt} $\kappa(X)=-1$
\begin{enumerate}
\item $\mathbb{P}^2$, $\mathbb{P}(\mathcal{O}_{\mathbb{P}^1}\oplus\mathcal{O}_{\mathbb{P}^1}(n))$, $n\geq0, n\neq1$.\\
\item\vspace{-10pt} Ruled surfaces over an ordinary elliptic curve $C$.
\end{enumerate}
\end{enumerate}
\end{theorem}

\smallskip The proof of nonexistence of liftings of Frobenius for those minimal surfaces not appearing in the above list is done by contradiction, which is drawn out from a morphism induced by a lifted Frobenius, see Proposition \ref{gb}. The positive part of this theorem is obtained either by proving the vanishing of the obstruction space (1)(b), or constructing the liftings directly (2)(b).

\medskip In the preliminary section we give some definitions and prove some lemma which are necessary for later development. The contents of the remaining section gives the proof our main result. The division of this section is partially based on the classification of minimal algebraic surfaces in characteristic $p$, see for instance \cite[Appendix A]{beauville1996complex}.

\subsection*{Conventions and Notations}\hfill
\begin{enumerate}[]
\item $k$: an algebraically closed field of characteristic $p>0$.
\item $W(k)$: the Witt ring with coefficients in $k$.
\item $W_n(k)$: the truncated Witt ring of length $n$ with coefficients in $k$.
\item $\kappa(X)$: the Kodaira dimension of $X$.
\item $h^i(\mathcal{F})=\dim H^i(X,\mathcal{F})$.
\end{enumerate}

\section{Preliminaries}\hfill

\noindent The main objective of this section is to fix terminologies and prove several results, among which Proposition \ref{gb} will be used later as an obstruction for the existence of liftings of Frobenius. Along the way, we will also review some facts on deformation theory.

\begin{definition} Let $X$ be a smooth variety defined over $k$, a flat lifting of $X$ to $W_n(k)$ \textnormal{(}resp. $W(k)$\textnormal{)} is a scheme $\boldsymbol{X}$ which is flat over $W_n(k)$ \textnormal{(}resp. $W(k)$\textnormal{)} such that $\boldsymbol{X}\times_{W_n(k)}k\cong X$ \textnormal{(}resp. $\boldsymbol{X}\times_{W(k)}k\cong X$\textnormal{)}. $X$ is said to have a $W_n$-lifting \textnormal{(}resp. $W(k)$-lifting or a lifting to characteristic 0\textnormal{)} of its Frobenius, if there exists a flat lifting $\boldsymbol{X}$ of $X$ to $W_n(k)$ \textnormal{(}resp. $W(k)$\textnormal{)} and a morphism $F_{\boldsymbol{X}}:\boldsymbol{X}\rightarrow\boldsymbol{X}$ such that the following diagram is commutative

$$\xymatrix{X\ar[r]^i\ar[d]_{F_X}&\boldsymbol{X}\ar[d]^{F_{\boldsymbol{X}}}\\
            X\ar[r]^i            &\boldsymbol{X}
            }$$
\noindent where $i:X\rightarrow\boldsymbol{X}$ is the closed immersion defined by the ideal $(p)$.
\end{definition}

As the above definition manifests, to study liftability of Frobenius, the deformation theory of schemes and morphisms over local artin rings is an indispensable tool. Related results can be found in \cite[Theorem 5.9]{illusie2005grothendieck} and we reproduce it here for convenience.

\begin{pro}\label{dfm} Let $X$ be a $S$-scheme and $j: X_0\rightarrow X$ be a closed immersion defined by an ideal $J$ such that $J^2=0$. Let $Y$ be a smooth $S$-scheme and $g: X_0\rightarrow Y$ be an $S$-morphism. There is an obstruction
$o(g,j)\in H^1(X_0, J\otimes_{\mathcal{O}_{X_0}}g^*T_{Y/S})$ whose vanishing is necessary and sufficient for the existence of an S-morphism $h: X\rightarrow Y$
extending $g$, i.e. such that $hj=g$. When $o(g,j) = 0$, the set of extensions $h$ of $g$ is an
affine space under $H^0(X_0, J
\otimes_{\mathcal{O}_{X_0}}
g^*T_{Y/S})$.

\smallskip\noindent Let $S_0\rightarrow S$ be a closed immersion defined by an ideal $I$ of square 0. Let $X_0$ be an $S_0$-scheme, then there is an obstruction $o(X_0,i)\in H^2(X_0, f^*_0I\otimes T_{X_0/S_0})$ (where $f_0:X_0\rightarrow S_0$ is the structure morphism) whose vanishing is the sufficient and necessary condition for the existence of a deformation of $X_0$ over $S$. When $o(X_0,i)=0$, the set of isomorphic classes of such deformations is an affine space under $H^1(X_0, f^*_0I\otimes T_{X_0/S_0})$ and the automorphism group of a fixed deformation is an affine space under $H^0(X_0, f^*_0I\otimes T_{X_0/S_0})$.
\end{pro}

\begin{corollary}\label{lfec} Let $X, Y$ be smooth varieties over $k$ such that $Y$ is an \'{e}tale cover of $X$. If $F_X$ is liftable to $W_2(k)$, then $F_Y$ is liftable to $W_2(k)$.

\end{corollary}
\begin{proof} Let $\boldsymbol{X}$ be a lifting of $X$ to $W_2(k)$ such that $F_X$ is liftable to $\boldsymbol{X}$. By the second part of the above proposition, there exists a lifting $\boldsymbol{Y}\rightarrow\boldsymbol{X}$ of $Y$ as an $X$-scheme. In particular, $Y$ is liftable to $W_2(k)$. To find a lifting of $F_Y$ to $\boldsymbol{Y}$, it suffices to find a lifting of the relative Frobenius $F_{Y/X}:Y\rightarrow Y\times_XX$ to a morphism $\boldsymbol{Y}\rightarrow \boldsymbol{Y}\times_{\boldsymbol{X}}\boldsymbol{X}$. By the first part of the above proposition, the obstruction space for such lifting is 0.

\end{proof}

The following useful lemma is self-evident and we omit its proof.

\begin{lemma}\label{mp} Let $M$ be a flat $\mathds{Z}/p^2\mathds{Z}$-module, then the `multiplication by $p$' map on $M$ defines an isomorphism
$$M/pM\xrightarrow{\boldsymbol{p}}pM.$$
\end{lemma}

\smallskip The following corollary can be deduced from proposition \ref{dfm} directly and we omit its proof.
\begin{corollary}\label{rcor} Let $A$ be commutative $k$-algebra, $\boldsymbol{A}$ be a flat lifting of $A$ over $W_2(k)$. Then for any two liftings $F_{\boldsymbol{A}}$, $F_{\boldsymbol{A}}'$ of $F_A$, we have $F'_{\boldsymbol{A}}(a)=F_{\boldsymbol{A}}(a)+\boldsymbol{p}\eta(a)$, where $\eta:A\rightarrow A$ is a unique function satisfying
\begin{enumerate}
\item $\eta(a_1+a_2)=\eta(a_1)+\eta(a_2)$
\item $\eta(a_1a_2)=a_1^p\eta(a_2)+a_2^p\eta(a_1)$
\end{enumerate}
\end{corollary}

\smallskip The following result will be useful later.
\begin{corollary}\label{deg} Let $A$ be a commutative $k$-algebra, and $X=\mathbb{P}^1_A$. Then to give a lifting of $F_X$ over $W_2(k)$ is equivalent to give a lifting of $F_A$ over $W_2(k)$ together with a polynomial with coefficients in $A$ of degree $\leq 2p$.

\end{corollary}

\begin{proof}By the second part of Proposition \ref{dfm}, a flat lifting of $X$ over $W_2(k)$ is unique up to isomorphism. In particular, let $\boldsymbol{A}$ be a flat lifting of $A$ over $W_2(k)$, it suffices to prove the corollary for all lifting of $F_X$ to $\boldsymbol{X}=\mathbb{P}^1_{\boldsymbol{A}}$. For any lifting of $F_X$, there is an induced endomorphism of $\Gamma(\boldsymbol{X},\mathcal{O}_{\boldsymbol{X}})\cong\boldsymbol{A}$ hence a lifting of $F_A$. On the other hand, let $U=\mathrm{Spec}\,\boldsymbol{A}[x]$, $V=\mathrm{Spec}\,\boldsymbol{A}[y]$ be two open subschemes of $\boldsymbol{X}$ glued by the relation $xy=1$. Then a lifting $F_{\boldsymbol{X}}$ induces a lifting
on $U$, on which the image of $x$ under $F_{\boldsymbol{X}}$
can be written as $x^p+\boldsymbol{p}f$ with $f\in A[x]$. In order that this morphism can be extended to $V$, one checks easily $\deg f\leq2p$.

\end{proof}
In \cite[Appendix Proposition 1]{mehta1987varieties}, the authors give another version of obstruction for liftings of Frobenius. Now we recall their results.

\begin{pro}\label{MS} Let $X$ be a smooth variety, $(\boldsymbol{X}_n,F_{\boldsymbol{X}_n})/W_{n+1}(k)$ a lifting of $X$ to $W_{n+1}(k)$
together with a lifting of Frobenius. Then the obstruction to the existence of a pair $(\boldsymbol{X}_{n+1},F_{\boldsymbol{X}_{n+1}})$ over $W_{n+2}(k)$ consisting of a lifting $\boldsymbol{X}_{n+1}$ of $\boldsymbol{X}_n$ and a lifting of Frobenius $F_{\boldsymbol{X}_{n+1}}$ such that $F_{\boldsymbol{X}_{n+1}}|_{\boldsymbol{X}_n}\cong F_{\boldsymbol{X}_n}$ is given by a class in $H^1(X,T_X\otimes B_X^1)$, where $B^1_X:=F_{X*}\mathcal{O}_X/\mathcal{O}_X$. The various liftings form a principal homogeneous space under $H^0(X,T_X\otimes B_X^1)$.
\end{pro}

Note that the obstruction space above takes into consideration of pairs $(\boldsymbol{X}_{n+1},F_{\boldsymbol{X}_{n+1}})$ up to isomorphism, which is different from Proposition \ref{dfm}.

\medskip The liftability of Frobenius is a fairly strong property for an algebraic variety, one can show \cite[Proposition 8.6]{illusie2002frobenius} a proper smooth variety with liftable Frobneius is ordinary in the following sense.

\begin{definition}\label{ordi} Let $X$ be a proper, smooth variety over $k$, $X$ is said to be ordinary if $H^j(X,B^i_X)=0$ for all $(i,j)$, where $B^i_X=d\Omega^{i-1}_{X/k}$ is the $i$-th coboundary of the de Rham complex $\Omega^{\bullet}_{X/k}$.

\end{definition}

\bigskip Next we prove the main result of this section, which serves as an obstruction in disproving the existence of $W_2$-lifting of Frobenius for most surfaces.

\smallskip Let $X$ be a smooth variety over $k$, $\boldsymbol{X}$ be a flat lifting of $X$ over $W_2(k)$ and $F_{\boldsymbol{X}}$ be a lifting of the Frobenius of $X$ to $\boldsymbol{X}$. Then we will have an induced morphism
$$dF_{\boldsymbol{X}}:F_{\boldsymbol{X}}^*\Omega^1_{\boldsymbol{X}/W_2(k)}\rightarrow\Omega^1_{\boldsymbol{X}/W_2(k)}.$$

\noindent Since the reduction of the above morphism modulo $p$ is 0 and the sheaf $\Omega^1_{\boldsymbol{X}/W_2(k)}$ is $W_2(k)$-flat, $dF_{\boldsymbol{X}}$ factors through $p\Omega^1_{\boldsymbol{X}/W_2(k)}$. By lemma \ref{mp} we have an isomorphism of $\mathcal{O}_X$-modules $p\Omega^1_{\boldsymbol{X}/W_2(k)}\cong\Omega^1_{X/k}$. Thus $dF_{\boldsymbol{X}}$ induces an $\mathcal{O}_X$-linear morphism
\begin{equation}\varphi_{\scriptscriptstyle{F_{\boldsymbol{X}}}}:F_{X}^*\Omega^1_{X/k}\rightarrow\Omega^1_{X/k}.\label{dF/p}\end{equation}

\begin{pro}\label{gb} Let $X$ be a smooth variety over $k$ admitting a $W_2$-lifting of its Frobenius, then the morphism $\varphi_{\scriptscriptstyle{F_{\boldsymbol{X}}}}$ is generically bijective.
\end{pro}

\begin{proof} It is easy to see the claim in the proposition can be reduced to the affine case and it suffices to prove the determinant of $\varphi_{\scriptscriptstyle{F_{\boldsymbol{X}}}}$ is nonzero. To be more precise, let $X=\mathrm{Spec}\,A$, $X\rightarrow \mathbb{A}^n$ be an \'{e}tale cover and $\{dt_i\}_{1\leq i\leq n}$ be a basis of $\Omega^1_{X/k}$ with $t_i\in A$. Then to give a flat lifting of $X$ over $W_2(k)$ is equivalent to give a flat $W_2(k)$-algebra $\boldsymbol{A}$ such that $\boldsymbol{A}/p\boldsymbol{A}\cong A$. Now we choose a set of liftings $\boldsymbol{t}_i\in \boldsymbol{A}$ of $t_i$, $1\leq i\leq n$, then for any lifting $F_{\boldsymbol{X}}$ of $F_X$ to $\boldsymbol{X}=\mathrm{Spec}\,\boldsymbol{A}$ we can find $\boldsymbol{f}_i\in \boldsymbol{A}$ such that $$F_{\boldsymbol{X}}(\boldsymbol{t}_i)=\boldsymbol{t}_i^p+p\boldsymbol{f}_i.$$
\noindent Then the morphism $\varphi_{\scriptscriptstyle{F_{\boldsymbol{X}}}}$ with respect to the basis $\{dt_1,\cdots,dt_n\}$ is given by the matrix
\begin{equation}\label{matrix}\mathrm{Diag}(t^{p-1}_1,\cdots,t_n^{p-1})+\left(\frac{\partial f_i}{\partial t_j}\right),\end{equation}
\noindent where $f_i$ is the reduction of $\boldsymbol{f}_i$ modulo $p$. In order to prove the determinant of the above matrix is nonzero, we need th following lemma.
\begin{lemma}\label{ufe} Let $A$ be a regular local ring over $k$ and $\hat{A}$ be the completion of $A$ with respect to its maximal ideal. Then for any derivation $D\in\mathrm{Der}_k(A,A)$, we can find $\hat{D}\in\mathrm{Der}_k(\hat{A},\hat{A})$ such that $\rho(D(a))=\hat{D}(\rho(a))$, where $a\in A$ and $\rho:A\rightarrow\hat{A}$ is the natural inclusion.

\end{lemma}

\begin{proof} Let $d:A\rightarrow\Omega^1_{A/k}$ be the K\"{a}hler differential, by \cite[Definition 11.4, 12.2, Proposition 12.4]{kunz1986kahler} there exists a \emph{universally finite $\rho$-extension of $d$}. Moreover, the universally finite module of differentials is isomorphic to $\Omega_{A/k}\otimes_A\hat{A}$ and the differential $\hat{d}:\hat{A}\rightarrow\Omega_{A/k}\otimes_A\hat{A}$ is nothing but taking differentials
term-by-term then summing up the results. To be more explicit, we have the following commutative diagram
 $$\xymatrix{A\ar[d]^{d}\ar[r]^{\rho}&\hat{A}\ar[d]^{\hat{d}}\\
\Omega_{A/k}\ar[r]&\Omega_{A/k}\otimes_A\hat{A}}
$$
\noindent By the universal property for the pair $(\Omega_{A/k}\otimes_A\hat{A},\hat{d})$ \cite[Definition 11.4 (b)]{kunz1986kahler}, one can see easily the derivation $\hat{D}=D\otimes\mathrm{id}_{\hat{A}}$ satisfies the requirement in the lemma.

\end{proof}

\noindent Therefore, if we denote by $\hat{f}_i$ the image of $f_i$ in $\hat{A}$ and $\hat{\partial}_j\in \mathrm{Der}_k(\hat{A},\hat{A})$ be derivation associated to $\frac{\partial}{\partial t_j}$ as in lemma \ref{ufe}, then the image of the determinant of the matrix $(\ref{matrix})$ in $\hat{A}$ under the inclusion $\rho$ is nothing but
\begin{equation}\det\left(\mathrm{Diag}(t^{p-1}_1,\cdots,t_n^{p-1})+(\hat{\partial}_j(\hat{f}_i)\right).\label{det}\end{equation}

\smallskip To prove the proposition, it suffices to show the coefficient of the monomial $t^{p-1}_1\cdots t_n^{p-1}$ in the formal power series (\ref{det}) is 1. Note that the determinant (\ref{det}) is equal to

$$\sum_{1\leq N\leq n}\sum_{i_1<\cdots<i_N}t^{p-1}_{i_1}\cdots t^{p-1}_{i_N}\det M_{i_1,\cdots,i_N},$$

\noindent where $M_{i_1,\cdots,i_N}$ is the minor of $(\hat{\partial}_j(\hat{f}_i))$ obtained by deleting the $i_1$-th, $\cdots, i_N$-th rows and columns. Therefore, it suffices to prove for any $m$ power series $\hat{f}_i\in k[[t_1,\cdots,t_n]]$, $1\leq i\leq m$, $1\leq m\leq n$, the coefficient of $t^{p-1}_1\cdots t_m^{p-1}$ in $\det(\hat{\partial}_j(\hat{f}_i))$ is 0, where $1\leq j\leq m$.

\smallskip Since $\hat{f}_i\in k[[t_1,\cdots,t_n]]$, it can be written as (might not be uniquely)
\begin{equation}\label{sum}\sum_{0\leq s_1,\cdots,s_n\leq p-1}a_{i,s_1,\cdots,s_n}t_1^{s_1}\cdots t_n^{s_n}+\sum_{1\leq s\leq n}t_s^pg_{is},\ \ a_{i,s_1,\cdots, s_n}\in k,\ g_{is}\in k[[t_1,\cdots, t_n]].\end{equation}

\noindent Note that the terms $\hat{\partial}_j(t_s^pg_{is}), 1\leq i, j\leq m, 1\leq s\leq n$ have no contribution in computing the coefficient of $t_1^{p-1}\cdots t_m^{p-1}$ in $\det(\hat{\partial}_j(\hat{f}_i))$, we may just assume $g_{is}=0$, for all $1\leq i\leq m, 1\leq s\leq n$ in proving the claim above.

\smallskip By the multilinearity of determinant, we are reduced to prove the following claim.

\smallskip\noindent Let $f_i=t_1^{k_{i1}}\cdots t_n^{k_{in}}$, $1\leq i\leq n$ such that $0\leq k_{i1},\cdots,k_{in}\leq p-1$, then the coefficient of $t^{p-1}_1\cdots t_n^{p-1}$ in the determinant of $(\frac{\partial f_i}{\partial t_j})$ is 0.

\medskip\noindent By assumption, one checks easily
$$\det\left(\frac{\partial f_j}{\partial t_i}\right)=\det(k_{ij})\prod_{j=1}^nt^{s_j}_j,\hspace{10pt}s_j=-1+\sum_{i=1}^nk_{ij}.$$

\noindent Let $s_j=p-1$ for $1\leq j\leq n$, then $\sum_{i=1}^nk_{ij}=p$ consequently $\det(k_{ji})=0$. Thus the claim above hence the proposition is proved.

\end{proof}

\section{Proof of the Theorem}

\subsection*{Surfaces with $\kappa\geq1$}

\begin{pro} Let $X$ be a proper smooth variety over $k$ such that $\kappa(X)\geq1$, then $F_X$ cannot be lifted to $W_2(k)$.
\end{pro}

\begin{proof}Otherwise, by Proposition \ref{gb} we will have a generically bijective morphism $\varphi_{\scriptscriptstyle{F_{\boldsymbol{X}}}}:F_X^*\Omega^1_{X/k}\rightarrow\Omega^1_{X/k}$. By taking determinant, we obtain an injective morphism $\wedge\varphi_{\scriptscriptstyle{F_{\boldsymbol{X}}}}:\omega^{\otimes p}_{X/k}\rightarrow\omega_{X/k}$. Iterating this morphism with its pullbacks via Frobenius, we get an injective morphism from $\omega^{\otimes p^n}_{X/k}$ to $\omega_{X/k}$ hence a nonzero global section of $\omega^{\otimes (1-p^n)}_{X/k}$ for all $n\geq1$. This is impossible when $\kappa(X)\geq1$, since a sufficiently high positive power of the canonical divisor is linearly equivalent to an effective divisor.

\end{proof}

\subsection*{Surfaces with $\kappa=0$}
\subsubsection*{K3 surface, Enriques surfaces and Quasi-hyperelliptic Surfaces}\hfill
\begin{pro}\label{toris} If $X$ has a torsion canonical line bundle and admits a lifting of $F_X$ over $W_2(k)$, then $\varphi_{\scriptscriptstyle{F_{\boldsymbol{X}}}}$ is an isomorphism. Moreover,
\begin{enumerate}
\item $\Omega^1_{X/k}$ is \'{e}tale trivializable,
\item $X$ contains no rational curves.
\end{enumerate}
\end{pro}

\begin{proof} Suppose we are given a lifting $F_{\boldsymbol{X}}$ of $F_X$, then by Proposition \ref{gb}, there is an induced generically bijective morphism $\varphi_{\scriptscriptstyle F_{\boldsymbol{X}}}$. By taking determinant, we have an injective morphism $\wedge\varphi_{\scriptscriptstyle F_{\boldsymbol{X}}}:\omega^{\otimes p}_{X/k}\rightarrow\omega^1_{X/k}$. Since $\omega^1_{X/k}$ is a torsion line bundle, $\wedge\varphi_{\scriptscriptstyle F_{\boldsymbol{X}}}$ must be an isomorphism. Thus $\varphi_{\scriptscriptstyle F_{\boldsymbol{X}}}$ is also an isomorphism and claim (1) follows readily from \cite[1.4 Satz]{lange1977vektorbundel}.

\smallskip If $X$ contains a rational curve, then we are given a nonconstant morphism $f:\mathbb{P}^1\rightarrow X$ by normalization. Moreover, we can assume $f$ is separable. Then $f$ induces a morphism
$$0\rightarrow T_{\mathbb{P}^1/k}\rightarrow f^*T_{X/k}.$$

\noindent Since any vector bundles on $\mathbb{P}^1$ decompose into direct sum of line bundles, we have
\begin{equation}\label{dmp}f^*T_{X/k}\cong \mathcal{O}_{\mathbb{P}^1}(2)\oplus\bigoplus_{i=1}^{n-1}\mathcal{O}_{\mathbb{P}^1}(n_i).\end{equation}
\noindent On the other hand, the bijective morphism $\varphi_{\scriptscriptstyle F_{\boldsymbol{X}}}$ induces the following isomorphism
$$f^*T_{X/k}\rightarrow f^*F^*_XT_{X/k}\cong F^*_{\mathbb{P}^1}f^*T_{X/k}.$$
\noindent However, this obviously contradicts the decomposition (\ref{dmp}).

\end{proof}

\begin{corollary} If $X$ is a K3 surface, an Enriques surface or a quasi-hyperelliptic surface, then the Frobenius morphism $F_X$ cannot be lifted to $W_2(k)$.
\end{corollary}
\begin{proof} Note that the canonical line bundles of all these surfaces are torsion, it suffices prove they do not satisfy one of the properties in
Proposition \ref{toris}. If $X$ is a K3 surface, it doesn't satisfy (1) since it is simply connected and has nontrivial tangent bundle. By a recent result \cite[Proposition 17]{bogomolov2011constructing}, a K3 surface doesn't satisfy property (2) either.

\smallskip If $X$ is an Enriques surface in characteristic $p\neq2$ or a singular Enriques surface in characteristic 2 \cite[Theorem 2.7]{crew1984etale}, then it has an \'{e}tale cover by a K3 surface. If $X$ is is a classical or supersingular Enriques surface in characteristic 2, then it is unirational by \cite[Theorem 2]{blass1982unirationality}. Therefore, all Enriques surfaces do not satisfy property (2).

\smallskip For all quasi-hyperelliptic surfaces, there exist a fibration by cuspidal rational curves \cite[Proposition 5]{BMII1977} hence they do not satisfy property (2).

\end{proof}

\subsubsection*{Hyperelliptic Surfaces}\hfill

\medskip\noindent A complete classification of hyperelliptic surfaces is given in the list \cite[page 37]{BMII1977}, in which each surface is isomorphic to a quotient of the product of two elliptic curves by a finite subgroup scheme. Moreover, the canonical line bundle of a hyperelliptic surface is torsion of order equals to $1,2,3,4$ or $6$ \cite[page 37]{BMII1977}.

By Corollary \ref{lfec}, if $X$ admits a lifting of its Frobenius to $W_2(k)$, then the aforementioned elliptic curves are both ordinary since their Frobenius can be lifted to $W_2(k)$. Therefore, if $X$ is of type b) c) or d) in characteristic $2$ or $3$ then $F_X$ cannot be lifted, since one of the two elliptic curves is supersingular.

\begin{lemma}\label{dolga} Let $f:Y\rightarrow X$ be a Galois cover belonging to one of the following types
\begin{enumerate}
\item the Galois group is $\mathds{Z}/2\mathds{Z}$ and $p\neq2$, $2\:|\:p-1$,
\item the Galois group is $\mathds{Z}/3\mathds{Z}$ and $p\neq3$, $3\:|\:p-1$,
\item the Galois group is $\mathds{Z}/4\mathds{Z}$ and $p\neq2$, $4\:|\:p-1$.
\end{enumerate}
Suppose $h^i(B^1_Y)=0$ for all $i$, then $h^i(B^1_X)=0$ for all $i$.

\end{lemma}

\begin{proof} It is easy to see we need only to treat the case when the cover is nontrivial. Suppose we are in case (1), then by \cite[proposition 0.1.6, 0.1.8]{cossec1989enriques}, we always have $f_*\mathcal{O}_Y\cong\mathcal{O}_X\oplus\mathcal{L}$, where $\mathcal{L}$ is a torsion line bundle on $X$ of order $2$. Therefore, \begin{equation}\notag f_*F_{Y*}\mathcal{O}_Y\cong F_{X*}f_*\mathcal{O}_Y\cong F_{X*}\mathcal{O}_X\oplus F_{X*}\mathcal{L}.\end{equation}

We claim the images of sub line bundles $\mathcal{O}_X$ and $\mathcal{L}$ of $f_*\mathcal{O}_Y$ under the inclusion $f_*\mathcal{O}_Y\hookrightarrow f_*F_{Y*}\mathcal{O}_Y$ fall in $F_{X*}\mathcal{O}_X$ and $F_{X*}\mathcal{L}$ respectively.
Indeed, since $\mathcal{L}$ is nontrivial, we have $H^0(F_{X*}\mathcal{L})=0$ hence there exists no nonzero morphism from $\mathcal{O}_X$ to $F_{X*}\mathcal{L}$. On the other hand, let $\mathcal{L}\rightarrow F_{X*}\mathcal{O}_X$ be a nonzero morphism, then by adjointness, we are given a nonzero morphism $F_X^*\mathcal{L}\rightarrow\mathcal{O}_X$. As $\mathrm{char}k$ is odd, $F_X^*\mathcal{L}\cong \mathcal{L}$, thus we are given a nonzero morphism $\mathcal{L}\rightarrow\mathcal{O}_X$, again impossible. Therefore, the quotient $B^1_X\cong F_{X*}\mathcal{O}_X/\mathcal{O}_X$ is a direct summand of $f_*B^1_Y=f_*F_{Y*}\mathcal{O}_Y/f_*\mathcal{O}_Y$ hence the lemma is proved in this case.

\smallskip The remaining cases can be proved similarly by using \cite[Proposition 7.1]{miranda1985triple} and \cite[7.2 Case 1]{hahn1996quadruple}. It suffices to observe $f_*\mathcal{O}_Y$ decomposes into direct sum of pairwise non-isomorphic torsion line bundles with order a factor of 3 and 4 respectively.

\end{proof}

\begin{pro} Let $X$ be a hyperelliptic surface of type b), c) or d) in characteristic $p\neq2,3$, then the Frobenius $F_X$ can be lifted to $W_2(k)$ if and only if

 \begin{enumerate}
 \item the associated elliptic curves $E_0, E_1$ are ordinary,
 \item $\omega^{\otimes(p-1)}_{X/k}\cong\mathcal{O}_X$.
 \end{enumerate}

\end{pro}
\begin{proof} The necessity of condition (1) follows from Proposition \ref{lfec} and necessity of condition (2) is contained in the proof of Proposition \ref{toris}. Next we prove the sufficiency. By Proposition \ref{MS}, it suffices to show the obstruction space $H^1(T_X\otimes B^1_X)$ is $0$. Since the tangent bundle $T_X$ splits as $T_X\cong\mathcal{O}_X\oplus\omega^{-1}_{X/k}$ \cite[page 495]{lang1979quasi}, thus we are reduced prove $h^1(B^1_X)=h^1(\omega^{-1}_{X/k}\otimes B^1_X)=0$.

\smallskip For any hyperelliptic surface of type b), c), or d), one can find an abelian surface or a hyperelliptic surface $Y$ and a cyclic Galois cover $Y\rightarrow X$ of order 2, 3, or 4, see \cite[page 37]{BMII1977}. Then by the orders of $\omega_{X/k}$ listed on \cite[page 37]{BMII1977} and condition (2) of this proposition, we have the divisibility conditions required in lemma \ref{dolga}. Therefore, the vanishing of $H^1(B^1_X)$ follows from lemma \ref{dolga}.

\smallskip On the other hand, if $X$ is of type b), c) or d), by \cite[Theorem 4.9]{lang1979quasi} we have $h^1(\omega^{-1}_{X/k})=h^2(\omega^{-1}_{X/k})=0$. By projection formula and condition (2) we obtain $h^1(F_{X*}\mathcal{O}_X\otimes\omega^{-1}_{X/k})=h^1(F_X^*\omega^{-1}_{X/k})=h^1(\omega^{-1}_{X/k})=0$. Therefore, by studying the long exact sequence derived from
$$0\rightarrow\omega^{-1}_{X/k}\rightarrow F_{X*}\mathcal{O}_X\otimes\omega^{-1}_{X/k}\rightarrow B_X^1\otimes\omega^{-1}_{X/k}\rightarrow0,$$
\noindent one can see easily $h^1(\omega^{-1}_{X/k}\otimes B_X^1)=0$ hence the proposition is proved.

\end{proof}
\begin{pro}\label{typea} Let $X$ be a hyperelliptic surface of type a) in characteristic $p\neq2$ such that the associated elliptic curves $E_0, E_1$ are ordinary, then the Frobenius $F_X$ can be lifted to $W_2(k)$.

\end{pro}
\begin{proof} By Proposition \ref{MS} it suffices to prove $H^1(T_X\otimes B_X)=0$. In this case we still have $T_X\cong\mathcal{O}_X\oplus\omega_{X/k}$ and $\omega_{X/k}$ is of order 2 \cite[Theorem 4.9]{lange1977vektorbundel}. Thus we are reduced to prove $h^1(B^1_X)=h^1(B^1_X\otimes\omega_{X/k})=0.$ By lemma \ref{dolga}, one can see easily $h^1(B^1_X)=0$ in both cases a1) and a2).

\smallskip Now we prove $h^1(B^1_X\otimes\omega_{X/k})=0$. Since $\omega_{X/k}$ is of order 2 and $p$ is odd, we have $F_{X*}\mathcal{O}_X\otimes\omega_{X/k}\cong F_{X*}(F^*_X\omega_{X/k})\cong F_{X*}\omega_{X/k}.$
\noindent Therefore, $B^1_X\otimes\omega_{X/k}\cong F_{X*}\omega_{X/k}/\omega_{X/k}$. Let $Y=E_0\times E_1$ and $f:Y\rightarrow X$ be the \'{e}tale cover described on \cite[page 37]{BMII1977}, if we can prove $F_{X*}\omega_{X/k}/\omega_{X/k}$ is isomorphic to a direct summand of $f_*B^1_Y$ then we are done.

\smallskip If $X$ is of type a1), then by \cite[Propostion 0.1.3]{cossec1989enriques}, we have
\begin{equation}\label{pbc}\omega_{Y/k}\cong f^*(\omega_{X/k}\otimes L^{-1}),\end{equation}
\noindent where $L$ is the line bundle such that $f_*\mathcal{O}_Y\cong\mathcal{O}_X\oplus L$. On the other hand since $f$ is \'{e}tale, $\omega_{Y/k}\cong f^*\omega_{X/k}$ hence $f^*L\cong f^*\omega_{X/k}$ is trivial. By \cite[Theorem 2.1]{jensen1978picard} we have $L$ is isomorphic either to $\mathcal{O}_X$ or $\omega_{X/k}$. Note the cover $f$ is nontrivial, hence $L$ is also nontrivial so $\omega_{X/k}\cong L$. By the proof of lemma \ref{dolga}, we can see $F_{X*}L/L$ is a direct summand of $f_*B^1_{Y}$, hence so is true for $F_{X*}\omega_{X/k}/\omega_{X/k}$.

\smallskip If $X$ is of type a2), let $Z=Y/\mathds{Z}/2\mathds{Z}$ and $g:Y\rightarrow Z$, $h:Z\rightarrow X$ be the intermediate \'{e}tale covers such that $f=hg$. Moreover, let $g_*\mathcal{O}_Y\cong\mathcal{O}_Z\oplus L_Z$ and $h_*\mathcal{O}_Z\cong\mathcal{O}_X\oplus L_X$. Then by the proof for type a1) case, we have $L_Z$ satisfies $L_Z\cong\omega_{Z/k}\cong h^*\omega_{X/k}$. Therefore, we have the following decomposition
$$f_*\mathcal{O}_Y\cong \mathcal{O}_X\oplus L_X\oplus\omega_{X/k}\oplus\omega_{X/k}\otimes L_X$$
\noindent Moreover, by formula \ref{pbc}, $\omega_{X/k}\ncong L_X$ since $\omega_{Y/k}$ is nontrivial. Therefore, the four line bundles on the right side are pairwisely non-isomorphic. One the other hand, as these line bundles are of order 1 or 2, one can prove similarly as in lemma \ref{dolga} that $F_{X*}\omega_{X/k}/\omega_{X/k}$
is a direct summand of $F_{X*}f_*\mathcal{O}_Y/f_*\mathcal{O}_Y\cong f_*B^1_Y$.

\end{proof}

\begin{pro} Let $X$ be a hyperelliptic surface of type a) in characteristic 2, then $F_X$ is liftable to $W_2(k)$ if and only if $E_0, E_1$ are ordinary.
\end{pro}
\begin{proof}If $X$ is of type a1) and satisfies the remaining conditions in the proposition, then by \cite[Theorem 2, Lemma 1.1]{mehta1987varieties} $X$ is ordinary. In particular $H^1(B^1_X)=0$. Note that in this case the tangent bundle of $X$ is trivial, hence by Proposition \ref{MS} the Frobenius $F_X$ can be lifted to $W_2(k)$.

\smallskip Now assume $X$ is of type a3), and the associated elliptic curves $E_0, E_1$ are ordinary. Let $Y=E_0\times E_1/(\mathds{Z}/2\mathds{Z})$, then $Y$ is a $\mu_{2,k}$-torsor over $X$ in the flat topology. By \cite[Proposition 0.18]{cossec1989enriques}, we have $f_*\mathcal{O}_Y\cong\mathcal{O}_X\oplus L$, where $L$ is nontrivial torsion line bundle of order 2. As in the proof of lemma \ref{dolga}, $f_*F_{Y*}\mathcal{O}_Y$ factors as $F_{X*}\mathcal{O}_X\oplus F_{X*}L$. One can prove easily the images of the sub line bundles $\mathcal{O}_X$ and $L$ under the inclusion $f_*\mathcal{O}_Y\hookrightarrow f_*F_{Y*}\mathcal{O}_Y$ now both fall in $F_{X*}\mathcal{O}_X$. Then we have
$$f_*B^1_Y\cong f_*F_{Y*}\mathcal{O}_Y/f_*\mathcal{O}_Y\cong F_{X*}\mathcal{O}_X/(\mathcal{O}_X\oplus L)\bigoplus F_{X*}L.$$

By the former part of this proof, $Y$ is ordinary hence $h^i(B_Y^1)=0$ for all $i$. Thus $H^i(F_{X*}\mathcal{O}_X/(\mathcal{O}_X\oplus L))=H^i(F_{X*}L)=0$ for all $i$.
On the other hand, we have the following exact sequence
$$0\rightarrow L\rightarrow F_{X*}\mathcal{O}_X/\mathcal{O}_X\rightarrow F_{X*}\mathcal{O}_X/(\mathcal{O}_X\oplus L)\rightarrow0.$$
\noindent Then after taking long exact sequence one gets $H^i(B^1_X)=0$ for all $i$. Since the tangent bundle of $X$ in this case is a nontrivial extension of $\mathcal{O}_X$ by itself \cite[Theorem 4.11]{lang1979quasi}, one has $H^i(T_X\otimes B^1_X)=0$, hence by proposition \ref{MS} $F_X$ is also liftable in this case.

\smallskip The necessity of the ordinarity of $E_0, E_1$ for case a1) follows from Corollary \ref{lfec}. For case a3) let $Z=E_0\times E_1/\mu_{2,k}$, then we have an \'{e}tale cover $Z\rightarrow X$. Again by Corollary \ref{lfec} $Z$ has also a liftable Frobenius. Thus $Z$ hence $E_0, E_1$ are ordinary.

\end{proof}

To make things more clear, we list the results above ($\checkmark$ for existence of lifting of Frobenius under suitable condition and $\times$ for nonexistsnce) in the following table.

\begin{table}[h]
\begin{center}
\begin{tabular}{ l r r r }

Case & $\mathrm{char}\neq2,3$  & $\mathrm{char}=3$ & $\mathrm{char}=2$ \\ \hline
a    & $\checkmark$           & $\checkmark$            & $\checkmark$    \\
b    & $\checkmark$           & $\times$                & $\times$\\
c    & $\checkmark$           & $\times$                & $\times$\\
d    & $\checkmark$           & $\times$                & $\times$
\end{tabular}

\caption{Hyperelliptic Surfaces}\label{Ta:first}\label{hyes}
\end{center}
\end{table}

\subsection*{Ruled Surfaces}\hfill

\medskip\noindent  The following proposition is the first step in proving the ruled surface part of Theorem \ref{mt}. We will prove it by using deformation theory directly.

\begin{pro}\label{bc}Let $X$ be a ruled surface over a curve $C$, if $F_X$ is liftable to $W_2(k)$, so is $F_C$. In particular, $C$ must be an ordinary elliptic curve or the projective line.\end{pro}

\smallskip Before proving this proposition, we first recall some facts on ruled surfaces. Let $X$ be a ruled surface over a smooth projective curve $C$, then by \cite[V, Proposition 2.2]{hartshorne1977algebraic} $X$ is isomorphic to the projective bundle $\mathbb{P}(E)$, where $E$ is a rank 2 vector bundle $E$ on $C$. Since any rank 2 vector bundle on a curve is an extension by line bundles \cite[V, Corollary 2.7]{hartshorne1977algebraic} and $\mathbb{P}(E)\cong\mathbb{P}(E\otimes N)$ for any line bundle $N$, we can assume $E$ fits into the following exact sequence
\begin{equation}\label{ex3}0\rightarrow\mathcal{O}_C\rightarrow E\rightarrow L\rightarrow0\end{equation}
\noindent on $C$, where $L$ is a line bundle on $C$.
Let $U,V$ be two open subsets of $C$ over which the vector bundle $E$ is trivialized by $$\alpha_{\scriptscriptstyle U}:E|_U\cong\mathcal{O}_Ue\oplus\mathcal{O}_Uf_{U}\quad \mathrm{and}\quad \alpha_{\scriptscriptstyle V}:E|_V\cong\mathcal{O}_Ve\oplus\mathcal{O}_Vf_V,$$

\noindent where $f_U\in E_U$, $f_V\in E_V$ and $e$ is the global section of $E$ defined by the exact sequence $(\ref{ex3})$. Then we have the following induced isomorphisms over $U$ and $V$
respectively
\begin{equation}\label{gmiso}\beta_{\scriptscriptstyle U}:X_U:=\pi^{-1}(U)\cong\mathbf{Proj}\ \mathcal{O}_U[e,f_U],\quad\beta_{\scriptscriptstyle V}:X_V:=\pi^{-1}(V)\cong\mathbf{Proj}\ \mathcal{O}_V[e,f_V].\end{equation}


\smallskip Now let $t=\frac{e}{f_U}$, $s=\frac{e}{f_V}$, $x=\frac{f_U}{e}$ and $y=\frac{f_V}{e}$, then it is easy to see $X$ is covered by four open affine subschemes isomorphic to $\mathrm{Spec}\,\mathcal{O}_U[t]$, $\mathrm{Spec}\,\mathcal{O}_V[s]$, $\mathrm{Spec}\,\mathcal{O}_U[x]$ and $\mathrm{Spec}\,\mathcal{O}_V[y]$ respectively. Suppose the transition matrix of $E$ from the base $\{f_U,e\}$ to $\{f_V,e\}$ over $U\cap V$ is given by $\binom{a\ b}{0\ 1}$, where $a, b\in\mathcal{O}_{U\cap V}$ and $a$ is invertible. Then the gluing rule between the four open subschemes is given by $tx=sy=1$ and $x=ay+b$.

\begin{proof}[Proof of Proposition \ref{bc}] Given a lifting $\boldsymbol{X}$ of $X$ to $W_2(k)$, we claim $\boldsymbol{X}$ induces a lifting $\boldsymbol{C}$ of $C$ to $W_2(k)$. This is equivalent to say any lifting of $X$ as a $k$-scheme is a lifting of $X$ as $C$-scheme. By Proposition \ref{dfm}, there is a one-to-one correspondence between the above two spaces of liftings and the vector spaces $H^1(X,T_{X/C})$ and $H^1(X,T_{X/k})$ respectively. On the other hand, there is a one-to-one correspondence between the space of liftings of $C$ to $W_2(k)$ and $H^1(C,T_{C/k})$. Therefore, it's enough to show $h^1(T_{X/C})+h^1(T_{C/k})\geq h^1(T_{X/k})$. This inequality follows easily from the long exact sequence derived from
$$0\rightarrow T_{X/C}\rightarrow T_X\rightarrow \pi^*T_C\rightarrow0.$$

Let $Y_U$ (resp. $Y_V$) be the open subscheme of $X_U$ (resp. $X_V$) isomorphic to $\mathrm{Spec}\,k[x]$ (resp. $\mathrm{Spec}\,k[y]$).
Let $\boldsymbol{Y_U}$ (resp. $\boldsymbol{Y_V}$) be the open subschemes of $\boldsymbol{X}$ with the underlying open subset $i(Y_U)$ (resp. $i(Y_V$)), where $i$ is the homeomorphism $X\rightarrow\boldsymbol{X}$. Let $\boldsymbol{C}$ be the lifting of $C$ induced by $\boldsymbol{X}$, then the gluing rule between the two open subschemes $\boldsymbol{Y_U}$, $\boldsymbol{Y_V}$ can be written as
\begin{equation}\label{trs}\boldsymbol{x}\mapsto \boldsymbol{ay}+\boldsymbol{b}+\boldsymbol{p}y^2f,\end{equation}

\noindent where $\boldsymbol{a}, \boldsymbol{b}\in \mathcal{O}_{\boldsymbol{C}}$ are liftings of $a, b$ and $f\in \mathcal{O}_{U\cap V}[y]$.

\smallskip Suppose we are given a lifting $F_{\boldsymbol{X}}$ of $F_X$, then $F_{\boldsymbol{X}}$ induces a lifting of the Frobenius of any open subscheme of it. In particular, we are given a lifting $F_{\boldsymbol{Y_U}}$ of $F_{Y_U}$ (resp. $F_{\boldsymbol{Y_V}}$ of $F_{Y_V}$). Now given $\boldsymbol{\lambda}\in\mathcal{O}_{\boldsymbol{U}}$, then the image of $\boldsymbol{\lambda}$ under $F_{\boldsymbol{Y_U}}$ can be written uniquely as $\sum_iF_i(\boldsymbol{\lambda})\boldsymbol{x}^i,\ F_i(\boldsymbol{\lambda})\in\mathcal{O}_{\boldsymbol{U}}.$ Moreover, for $i\geq1$, we have $pF_i(\boldsymbol{\lambda})=0$. One can see easily the function $F_0:\mathcal{O}_{\boldsymbol{U}}\rightarrow \mathcal{O}_{\boldsymbol{U}}$ defines a lifting of $F_U$. Similarly, let $\boldsymbol{\mu}\in\mathcal{O}_{\boldsymbol{V}}$, if we write the image of $\boldsymbol{\mu}$ under $F_{\boldsymbol{X}}$ as
$\sum_iG_i(\boldsymbol{\mu})\boldsymbol{y}^i$ with $\ G_i(\boldsymbol{\mu})\in\mathcal{O}_{\boldsymbol{V}}$, then $G_0$ defines a lifting of $F_V$.

\smallskip Now given a section $\lambda\in\mathcal{O}_{\boldsymbol{U}\cap\boldsymbol{V}}$, by (\ref{trs}) we have
$$\sum_iF_i(\boldsymbol{\lambda})(\boldsymbol{ay+b}+\boldsymbol{p}y^2f)^i=\sum_iG_i(\boldsymbol{\lambda})\boldsymbol{y}^i$$
\noindent By comparing the terms with degree 0 on the two sides we get \begin{equation}\label{trs2}\sum_iF_i(\boldsymbol{\lambda})\boldsymbol{b}^i=\sum_iG_i(\boldsymbol{\lambda}).\end{equation}
\noindent As we note above, $pF_i(\boldsymbol{\lambda})=0$ so $p\sum_{i\geq1}F_i(\boldsymbol{\lambda})\boldsymbol{b}^i=0$ hence the left side of the above equality can be written as $F_0(\boldsymbol{\lambda})+\boldsymbol{p}\eta(\lambda)$, where the function $\eta:\mathcal{O}_{U}\rightarrow\mathcal{O}_{U}$ is uniquely defined by $\boldsymbol{p}\eta(\lambda)=\sum_{i\geq1}F_i(\boldsymbol{\lambda})\boldsymbol{b}^i$. Moreover, as $F_0$ and $G_0$ are liftings of $F_{U\cap V}$, therefore the restriction of $\eta$ to $\mathcal{O}_{U\cap V}$ satisfies the condition listed in Corollary \ref{rcor}. As the base scheme $U$ is integral $\eta$ as a function from $\mathcal{O}_U$ to itself also satisfies these conditions. By applying Corollary \ref{rcor} again, $F_0(\boldsymbol{\lambda})+\boldsymbol{p}\eta(\lambda)$ is also a lifting of $F_U$ hence we are given a lifting of $F_C$ and the proposition is proved.

\end{proof}

\smallskip  If the base curve is $\mathbb{P}^1$, then $X$ is a toric surface, hence endowed with a lifting of Frobenius to $W_2(k)$. Next we consider the case when $C$ is an ordinary elliptic curve.

\smallskip Given a lifting $\boldsymbol{C}$ of $C$ to $W_2(k)$ and a lifting $\boldsymbol{E}$ of $E$ over $\boldsymbol{C}$, then it is easy to see $\boldsymbol{X}:=\mathbb{P}(\boldsymbol{E})$ is a lifting of $X$. Next we will prove the Frobenius of $X$ can be lifted to $\boldsymbol{X}$.

\begin{pro} Let $C$ be an ordinary elliptic curve over $k$, $E$ be a vector bundle of rank 2 over $C$, and $X=\mathbb{P}(E)$ be the associated ruled surface, then $F_X$ can be lifted to $W_2(k)$

\end{pro}

\begin{proof} Following the notations used earlier in this subsection, let $\boldsymbol{U}, \boldsymbol{V}$ be open subschemes of $\boldsymbol{X}$ with underlying open subset $U$ and $V$, $\boldsymbol{e}, \boldsymbol{f_U}, \boldsymbol{f_V}$ be sections of $\boldsymbol{E}$ lifting $e, f_U$ and $f_V$
respectively. Furthermore, let $\boldsymbol{t}=\frac{\boldsymbol{e}}{\boldsymbol{f_U}}$, $\boldsymbol{s}=\frac{\boldsymbol{e}}{\boldsymbol{f_V}}$, $\boldsymbol{x}=\frac{\boldsymbol{f_U}}{\boldsymbol{e}}$ and $\boldsymbol{y}=\frac{\boldsymbol{f_V}}{\boldsymbol{e}}$, then $\boldsymbol{X}$ is covered by the open subschemes $\mathrm{Spec}\ \mathcal{O}_{\boldsymbol{U}}[\boldsymbol{t}]$, $\mathrm{Spec}\ \mathcal{O}_{\boldsymbol{U}}[\boldsymbol{x}]$, $\mathrm{Spec}\ \mathcal{O}_{\boldsymbol{V}}[\boldsymbol{s}]$ and $\mathrm{Spec}\ \mathcal{O}_{\boldsymbol{V}}[\boldsymbol{y}]$. Moreover, the gluing rules between these open subschemes are given by $\boldsymbol{tx}=1$, $\boldsymbol{sy}=1$ and $\boldsymbol{x}=\boldsymbol{ay}+\boldsymbol{b}$.

\smallskip We fix a lifting $F_{\boldsymbol{C}}:\boldsymbol{C}\rightarrow \boldsymbol{C}$, and define a lifting of $F_X$ over $\mathrm{Spec}\,\mathcal{O}_{\boldsymbol{U}}[\boldsymbol{x}]$ by sending $\boldsymbol{x}$ to $\boldsymbol{x}^p$. We take $f$ to be 0, then by the relation $\boldsymbol{x}=\boldsymbol{ay}+\boldsymbol{b}$, the image of $\boldsymbol{y}$ under the lifted Frobenius is given by
$$F_{\boldsymbol{X}}(\boldsymbol{y})=F_{\boldsymbol{X}}\left(\frac{\boldsymbol{x-b}}{\boldsymbol{a}}\right)=\frac{(\boldsymbol{ay+b})^p-F_{\boldsymbol{C}}(\boldsymbol{b})}{F_{\boldsymbol{C}}(\boldsymbol{a})}.$$

\noindent The right side of the above equality can be written as $\boldsymbol{y}^p+\boldsymbol{p}h$, $h\in \mathcal{O}_{U\cap V}[y]$ and it is easy to see $\deg h\leq p$. By Corollary \ref{deg} this lifting of Frobenius over $\mathrm{Spec}\,\mathcal{O}_{\boldsymbol{U}}[\boldsymbol{x}]$ (resp. $\mathrm{Spec}\,\mathcal{O}_{\boldsymbol{V}}[\boldsymbol{y}]$) can be extended to $\mathrm{Spec}\,\mathcal{O}_{\boldsymbol{U}}[\boldsymbol{t}]$ (resp. $\mathrm{Spec}\,\mathcal{O}_{\boldsymbol{V}}[\boldsymbol{s}]$). It is easy to see these liftings glue well hence the proposition is proved.
\end{proof}

\medskip
\subsection*{Acknowledgement:} The author would like to thank Prof. Xiaotao Sun for a careful reading of a draft of this work.

\bibliographystyle{amsplain}
\bibliography{xhBibTex}

\end{document}